\documentclass[11pt,twoside,english]{article}
\usepackage[T1]{fontenc}
\usepackage[latin9]{inputenc}
\usepackage{amsmath}
\usepackage{amsthm}
\usepackage{amssymb}

\makeatletter
\theoremstyle{plain}
\newtheorem{thm}{\protect\theoremname}
  \theoremstyle{plain}
  \newtheorem{conjecture}[thm]{\protect\conjecturename}
  \theoremstyle{remark}
  \newtheorem*{acknowledgement*}{\protect\acknowledgementname}
  \theoremstyle{definition}
  \newtheorem{defn}[thm]{\protect\definitionname}
  \theoremstyle{plain}
  \newtheorem{prop}[thm]{\protect\propositionname}
  \theoremstyle{remark}
  \newtheorem{claim}[thm]{\protect\claimname}
\newcommand{\lyxaddress}[1]{
\par {\raggedright #1
\vspace{1.4em}
\noindent\par}
}

\newcommand{\hide}[1]{}

\DeclareMathOperator*{\dom}{dom}

\DeclareMathOperator{\bdim}{dim_\textrm{B}}

\DeclareMathOperator{\hdim}{dim_\textrm{H}}

\date{}

\makeatletter
\def\blfootnote{\xdef\@thefnmark{}\@footnotetext}
\makeatother

\makeatother

\usepackage{babel}
  \providecommand{\acknowledgementname}{Acknowledgement}
  \providecommand{\claimname}{Claim}
  \providecommand{\conjecturename}{Conjecture}
  \providecommand{\definitionname}{Definition}
  \providecommand{\propositionname}{Proposition}
\providecommand{\theoremname}{Theorem}

\begin{document}

\title{Smooth symmetries of $\times a$-invariant sets}

\author{Michael Hochman}
\maketitle
\begin{abstract}
We\blfootnote{Supported by ERC grant 306494 and ISF grant 1702/17}\blfootnote{\emph{2010 Mathematics Subject Classication}. 28A80, 11K55, 11B30, 11P70}
study the smooth self-maps $f$ of $\times a$-invariant sets $X\subseteq[0,1]$.
Under various assumptions we show that this forces $\log f'(x)/\log a\in\mathbb{Q}$
at many points in $X$. Our method combines scenery flow methods and
equidistribution results in the positive entropy case, where we improve
previous work of the author and Shmerkin, with a new topological variant
of the scenery flow which applies in the zero-entropy case. 
\end{abstract}

\section{Introduction}

For an integer $a\geq2$ let $T_{a}$ denote the self-map of $[0,1]$,
or $\mathbb{R}/\mathbb{Z}$, given by $T_{a}x=ax\bmod1$. Furstenberg
famously proved that, although each of these maps individually admits
a multitude of closed invariant sets, when $a,b$ are non-commensurable,
the only jointly invariant ones $T_{a},T_{b}$ are trivial  \cite[Theorem IV.1]{Furstenberg67}.
Here by trivial we mean either finite or all of $[0,1]$, and we say
that real numbers $a,b\neq0$ are commensurable, and write $a\sim b$,
if $\log|a|/\log|b|\in\mathbb{Q}$; otherwise they are non-commensurable,
denoted $a\not\sim b$.

Furstenberg's theorem has seen many generalizations. The main direction
of generalization has been to commuting actions in algebraic settings:
these include commuting automorphisms of compact abelian groups \cite{Berend1983,Berend1984},
and analogous (though still partial) measure-rigidity results in the
automorphism setting as well as for higher-rank diagonal flows on
homogeneous spaces  \cite{KatokSpatzier1996,EinsiedlerLindenstrauss2003,EinsiedlerKatokLindenstrauss2006}.
Another generalization is to commuting diffeomorphisms on compact
manifolds, see e.g. \cite{KalininKatokHertz2008}. These results are
distinct from the algebraic ones, although they share many common
methods and are related by conjectures predicting that, in many cases,
the only commuting smooth maps are those that are conjugate to algebraic
ones.

This paper deals with another related phenomenon, namely, that if
$X$ is $T_{a}$-invariant and non-trivial, then very few smooth maps
can map it onto (or even into) itself, and in fact, any such map must
locally behave like $T_{a}$, in the sense that it must satisfy $f'\sim a$
for ``most'' $x\in X$. We stress that no assumption of commutation
is made between $T_{a}$ and $f$. We also show that if $Y$ is another
 non-trivial $T_{b}$-invariant set for $b\not\sim a$, then $X$
cannot be mapped to $Y$ by a smooth map. 

Results of this kind have a long history for ``fractal'' sets, sometimes
even in the Lipschitz category. For example, when $X,Y$ are ``deleted-digit''
Cantor sets, i.e. defined by restricting individual digits in non-commensurable
bases $a$ and $b$, Falconer and Marsh showed that there is no bi-Lipschitz
map taking $X$ to $Y$ \cite{FalconerMarsh1992} (they have more
general results too but in another context). More recent work on self-similar
sets with similar flavor appears in the work of Elekes, Keleti and
M\'ath\'e, and of Feng, Huang and Rao \cite{Elekes09,FengHuangRao2014}.
Our results have some overlap with these, but we emphasize that we
deal with much more general sets, including sets of entropy (or dimension)
zero.

Recall that we say a set $X\subseteq[0,1]$ is \emph{non-trivial }if
it is infinite and $X\neq[0,1]$. 
\begin{thm}
\label{thm:joint-invariant-sets}Let $X,Y\subseteq[0,1]$ be non-trivial
and closed sets that are invariant under $T_{a},T_{b},$ respectively,
$a\not\sim b$. Then no $C^{2}$-diffeomorphism of $\mathbb{R}$ or
$\mathbb{R}/\mathbb{Z}$ can map $X$ onto $Y$. 
\end{thm}

For self-maps of a single $T_{a}$-invariant set we obtain results
under some additional dynamical and regularity assumptions. A set
is\emph{ perfect }if it has no isolated points. If $X$ is $T_{a}$-invariant
then a point $x\in X$ is \emph{transitive }if its orbit $\{T_{a}^{n}x\}_{n=1}^{\infty}$
is dense in $X$, and $X$ is \emph{transitive }if it contains a transitive
point. It is \emph{minimal }if every point is transitive. Minimal
infinite systems are perfect. Adapting Furstenberg's terminology,
we say that $X$ is \emph{self-restricted }if $X-X\neq[0,1]\bmod1$,
which holds in particular when $X$ is minimal (see \cite{Furstenberg67}),
or $\dim X<1/2$. Finally, we say that $f\in C^{1}$ is \emph{piecewise
curved }if $f'$ is piecewise strictly monotone (thus $f$ is locally
strictly concave or convex).
\begin{thm}
\label{thm:partial-invariance-under-f}Let $X\subseteq[0,1]$ be a
closed, perfect, transitive and non-trivial $T_{a}$-invariant set.
Then there is no piecewise-curved $f\in C^{1}$ that is differentiable
on $X$ and maps $X$ onto itself. Furthermore, without assuming curvedness,
we have:

\begin{enumerate}
\item \label{transitive case:affine}An affine map $f$ with $f(X)\subseteq X$
has $f'\sim a$.
\item \label{transitive-case:minimal}If $X$ is minimal and $f\in C^{1}$,
then $f(X)\subseteq X$ implies that $f'(x)\sim a$ for all $x\in X$;
in fact for any open interval $I$, if $I\cap X\neq\emptyset$ then
$f(X\cap I)\subseteq X$ implies $f'(x)\sim a$ for all $x\in X\cap I$.
\item \label{transitive-case:restricted}If $X$ is self-restricted and
$f\in C^{1}$ then $f(X)=X$ implies $f'(x)\sim a$ for all $x$ in
a dense, relatively open subset of $X$. 
\end{enumerate}
In particular, in (\ref{transitive-case:minimal}) and (\ref{transitive-case:restricted}),
if in addition $f$ is real-analytic then $f$ is affine.
\end{thm}

We do not know whether $f(X)\subseteq X$ for transitive $X$ implies
similar conclusions in general.

The proofs split into two parts. The main new ingredient of this paper
is a method to handle the case that $X$ is self-restricted (though
actually the main case of interested is the more special case when
its entropy is zero). To explain this part it is useful to recall
Furstenberg's original proof that there are no non-trivial jointly
$T_{a},T_{b}$-invariant sets. For such an $X$, the first observation
is that $X-X\bmod1$ is jointly $T_{a},T_{b}$ invariant (because
both $T_{a}$ and $T_{b}$ are endomorphisms of $\mathbb{R}/\mathbb{Z}$).
On the other hand, if $X$ is infinite, then it has an accumulation
point, whence $0$ is an accumulation point of $X-X$. From $a\not\sim b$
it follows that $\{a^{n}b^{m}\}_{m,n\in\mathbb{N}}$ is non-lacunary,
implying that for every $\varepsilon>0$ there is a $\delta>0$ such
that if $y\in[0,\delta)$ then $\{T_{a}^{i}T_{b}^{j}y\}_{i,j\in\mathbb{N}}=\{a^{i}b^{j}y\bmod1\}_{i,j\in\mathbb{N}}$
is $\varepsilon$-dense. Taken together, this shows that $X-X$ is
$\varepsilon$-dense in $[0,1]$ for every $\varepsilon$, so $X-X=[0,1]$.

In contrast, in our setting the first stage of this argument already
fails: we are assuming that $X$ is invariant under $T_{a}$ and under
another map $f$, and while $X-X$ is still invariant under $T_{a}$,
it is generally not $f$-invariant. The main new ingredient in our
proof is to use a local version of the difference set $X-X$ which
behaves well under smooth maps. This is developed in Section \ref{sec:local-distance-set},
and is motivated by the scenery flow and spectral arguments of \cite{Hochman2010,HochmanShmerkin2015-equidistribution-from-fractal-measures}.
These do not make a formal appearance here, but see the remark at
the end of Section \ref{sec:local-distance-set}.

In the non-restricted case, and specifically when $X$ has positive
topological entropy, the theorems above are proved via analysis of
invariant measures on $X$, adapting scenery flow methods from \cite{Hochman2010,HochmanShmerkin2015-equidistribution-from-fractal-measures}.
This is also the source of the piecewise curvedness requirement.\footnote{After this work was completed, P. Shmerkin \cite{Shmerkin2016} and
M. Wu \cite{Wu2016} independently proved a result on slices of products
of positive-dimension $T_{a}$- and $T_{b}$-invariant sets which
implies that $C^{1}$ is enough in the positive-dimension case of
Theorem \ref{thm:joint-invariant-sets}. The connection with Shmerkin
and Wu's papers is that any $C^{1}$-embedding $f:X\rightarrow Y$
implies that the curve $y=f(x)$ intersects $X\times Y$ in a set
diffeomorphic to $X$, while their results imply that such an intersection
must have dimension at most $\max\{0,\dim X+\dim Y-1\}$, which is
always $<\dim X$ when the latter is positive.} To state our result for positive-entropy measures, let us say that
a probability measure $\mu$ has dimension $t$ if $\mu(E)<1$ for
all Borel sets $E$ with $\dim E<t$, but it is supported on some
set of dimension $t$. We write $f\mu=\mu\circ f^{-1}$ for the push-forward
of $\mu$ by a map $f$, and note that when $f$ is bi-Lipschitz,
$\mu$ and $f\mu$ have the same dimension. Our result for measures
is the following:\footnote{We recently learned of results by Eskin, related to work of Brown
and Rodriguez-Hertz, which shows that certain ``general position'',
expanding-on-average pairs of diffeomorphisms of a compact manifold
can preserve nothing but Lebesgue measure. Their methods do not appear
to apply in our setting, where both the expansion and invertability
are absent.} 
\begin{thm}
\label{thm:measures}For every $T_{a}$-ergodic measure $\mu$ of
dimension $0<s<1$, there exists $\varepsilon=\varepsilon(\mu)$ such
that the following holds. For every piecewise curved $f$, every weak-{*}
accumulation point of the sequence $\frac{1}{N}\sum_{n=0}^{N-1}T_{a}^{n}(f\mu)$
has dimension at least $s+\varepsilon$.
\end{thm}

Combined with the variational principle, this implies that if $X$
is a non-trivial $T_{a}$-invariant set with positive dimension, then
$f(X)\subseteq X$ is impossible for a piecewise curved $f$. Thus
for a piecewise curved function $f$, every jointly $f$- and $T_{a}$-invariant
ergodic probability measure is either Lebesgue or has dimension zero,
and every jointly invariant set is either $[0,1]$ or has dimension
$0$. This and the other results stated earlier make the following
conjecture seems reasonable:
\begin{conjecture}
Fix $T_{a}$ and let $f\in C^{\omega}(\mathbb{R})$ or $f\in C^{\omega}(\mathbb{R}/\mathbb{Z})$,
and assume that $f$ is not affine. Then every jointly $T_{a}$- and
$f$-invariant set is trivial.
\end{conjecture}

Of course one could also ask this for $f$ with less smoothness, or
make the same conjecture for measures.

The paper is organized as follows. We begin with the topological analysis
of the zero entropy (or self-restricted) case: In Section \ref{sec:local-distance-set}
we define the local difference (or distance) set and discuss its properties,
in Section \ref{sec:Dimension} we discuss dimension and the relation
between the local difference set and the original set, and in Sections
\ref{sec:Proof-of-joint-invariance} and \ref{sec:Proof-of-partial-invariance}
we prove Theorem \ref{thm:joint-invariant-sets} and (most of) Theorem
\ref{thm:partial-invariance-under-f}. The last section is devoted
to Theorem \ref{thm:measures} and completing the proof of Theorem
\ref{thm:partial-invariance-under-f} (\ref{transitive case:affine}).
\begin{acknowledgement*}
I am grateful to J.-P. Thouvenot for encouraging me to revisit these
questions. I also thank Amir Algom and the anonymous referee for their
comments on an early version of the paper.
\end{acknowledgement*}

\section{\label{sec:local-distance-set}Localizing the difference set }

As explained in the introduction, we require a ``localized'' version
of the difference set $X-X$. For this, define the $a$-adic fractional
part of $s>0$ by 
\[
\{s\}_{a}=-\log_{a}s\bmod1
\]
i.e. if $0\leq t<1$ and $s=a^{-(k+t)}$ for $k\in\mathbb{N}$, then
$\{s\}_{a}=t$. For $s<0$ define $\{s\}_{a}=\{-s\}_{a}$.
\begin{defn}
\label{def:local-distance-set}Let $X\subseteq\mathbb{R}$. For $2\leq a\in\mathbb{N}$
and $x\in\mathbb{R}$ the $a$\emph{-adic local }difference\emph{
set }of $X$ at $x$ is the set 
\[
F_{a,X}(x)=\left\{ t\in\mathbb{R}/\mathbb{Z}\,\left|\begin{array}{c}
\exists x_{n},x'_{n}\in X\mbox{ s.t. }(x_{n},x'_{n})\rightarrow(x,x)\\
x_{n}\neq x'_{n}\mbox{ and }\{x_{n}-x'_{n}\}_{a}\rightarrow t
\end{array}\right.\}\right\} 
\]

For $X\subseteq\mathbb{R}/\mathbb{Z}$ we define $F_{a,X}(x)=F_{a,Y}(y)$
where $x=y\bmod1$ and $X=Y\bmod1$, and the lift $Y$ is chosen so
that reduction modulo $1$ is a bijection of neighborhoods of $x$
and $y$. When $X$ is $T_{a}$-invariant, we view $X$ as a subset
of $\mathbb{R}/\mathbb{Z}$ when defining $F_{a,X}(x)$, even if $X$
is initially given as a subset of $[0,1]$. This convention potentially
enlarges $F_{a,X}(0)$ when $0,1\in X$.
\end{defn}

This defines a function $F_{a,X}:X\rightarrow\{\mbox{subsets of }\mathbb{R}/\mathbb{Z}\}$,
but we sometimes think of the range as subsets of $[0,1]$, and make
the identification whenever convenient. It will be convenient to write
\[
F_{a,X}(X)=\bigcup_{x\in X}F_{a,X}(x)
\]

We state, mostly without proof, some elementary properties of $F_{a,X}(x)$.
\begin{enumerate}
\item (Locality): $F_{a,X}(x)=F_{a,X\cap B_{r}(x)}(x)$ for any $r>0$.
\item (Monotonicity): $Y\subseteq X$ implies $F_{a,Y}(x)\subseteq F_{a,X}(x)$. 
\item (Unions) If $X=\bigcup X_{i}$ then $F_{a,X}(X)\supseteq\bigcup F_{a,X_{i}}(X_{i})$.
If the $X_{i}$ are closed and disjoint then $F_{a,X}(X)=\bigcup F_{a,X_{i}}(X_{i})$.
\item (Non-triviality): $F_{a,X}(x)\neq\emptyset$ if and only if $x$ is
an accumulation point of $X$.
\item (Closure): $F_{a,X}(x)$ is closed for all $x$.
\item (Semi-continuity): If $x_{n}\rightarrow x$ and $t_{n}\in F_{a,X}(x_{n})$,
and if $t_{n}\rightarrow t$ in $\mathbb{R}/\mathbb{Z}$, then $t\in F_{a,X}(x)$
(equivalently, in the space of compact subsets of $\mathbb{R}/\mathbb{Z}$,
any sub-sequential limit $E$ of $F_{a,X}(x_{n})$ in the Hausdorff
metric satisfies $E\subseteq F_{a,X}(x)$).

In particular, if $X$ is compact, then $F_{a,X}(X)=\bigcup_{x\in X}F_{a,X}(x)$
is closed. 
\item \label{enu:eigenfunc}(Linear transformation): For $X\subseteq\mathbb{R}$
and any $0\neq b\in\mathbb{R}$, we have\footnote{We use the usual notation for arithmetic operations between sets:
for $u\in\mathbb{R}$ and $W\subseteq\mathbb{R}$, $uW=\{uw\,:\,w\in W\}$
and $W+u=\{w+u\,:\,w\in W\}$.} 
\[
F_{a,bX}(bx)=F_{a,X}(x)-\log_{a}|b|\bmod1
\]
\item ($T_{b}$-transformation): For $X\subseteq\mathbb{R}/\mathbb{Z}$
and $b\in\mathbb{N}$, 
\[
F_{a,T_{b}X}(T_{b}x)\supseteq F_{a,X}(x)-\log_{a}b\bmod1
\]
Indeed, for small $r>0$ the map $T_{b}$ acts on $B_{r}(x)$ as an
affine map of expansion $b$, and sends $X\cap B_{r}(x)$ into a subset
of $T_{b}X\cap B_{br}(T_{b}x)$. The inclusion above follows by locality
and monotonicity.

For $X\subseteq[0,1]$, the same holds except possibly at points $x\in\frac{1}{b}\mathbb{Z}$,
where $T_{b}$ is discontinuous. 
\item ($C^{1}$-transformation): If $f:\mathbb{R}\rightarrow\mathbb{R}$
is a $C^{1}$-diffeomorphism, then 
\[
F_{a,fX}(fx)=F_{a,X}(x)+\log_{a}|f'(x)|\bmod1
\]
Indeed, for $t\in F_{a,X}(x)$ let $x_{n},x'_{n}\in X$ converge to
$x$ and satisfy $\{x_{n}-x'_{n}\}_{a}\rightarrow t$. Then by calculus,
$f(x_{n})-f(x'_{n})=f'(\xi_{n})\cdot(x_{n}-x'_{n})$ for some point
$\xi_{n}$ intermediate between $x_{n},x'_{n}$, so $\{f(x_{n})-f(x'_{n})\}_{a}=\{x_{n}-x'_{n}\}_{a}-\log_{a}|f'(\xi_{n})|$.
Since $x_{n},x'_{n}$ converge to $x$, also $\xi_{n}\rightarrow x$,
so we have $f'(\xi_{n})\rightarrow f'(x)$ and $t+\log_{a}f'(x)=\lim\{f(x_{n})-f(x'_{n})\}_{a}\in F_{a,f(X)}(f(x))$.
The converse inclusion is shown similarly by considering $f^{-1}$. 
\end{enumerate}
Definition \ref{def:local-distance-set} is motivated by the spectral
analysis of the scenery flow of $\times a$-invariant measures that
was carried out in \cite{Hochman2010}. Indeed, property (\ref{enu:eigenfunc})
means that the function $t\mapsto F_{a,tX}(0)$ is periodic (in $\log t$),
and so defines an eigenfunction for the scenery flow (with values
in the space of closed subsets of $\mathbb{R}/\mathbb{Z}$).

\section{\label{sec:Dimension}Dimension}

We write $\dim X$ for the Hausdorff dimension of a set $X\subseteq\mathbb{R}$
or $\mathbb{R}/\mathbb{Z}$, and $\bdim X$ for its box dimension,
defined as the limit as $r\searrow0$ of $\log N(X,r)/\log(1/r)$,
$N(X,r)$ is the minimal number of $r$-balls needed to cover $X$
(in the cases we apply this, the limit exists). We have the following
standard facts.
\begin{enumerate}
\item $\bdim$ is stable under finite unions: $\bdim\bigcup_{i=1}^{n}X_{i}=\max_{1\leq i\leq n}\bdim X_{i}$.
\item $\dim$ and $\bdim$ are non-decreasing under Lipschitz maps. 
\item $\dim(X\times Y)\leq\dim X+\bdim Y$.
\item $\dim(X\times Y)\geq\dim X+\dim Y$, with equality if $\dim X=\bdim X$.
\item If $X$ is closed and $T_{a}$-invariant then $\bdim X$ exists and
$\bdim X=\dim X=h_{top}(X,T_{a})/\log a$, where $h_{top}(X,T_{a})$
is the topological entropy of $(X,T_{a})$ \cite[Section III]{Furstenberg67}.
\end{enumerate}
It follows from these properties that if $Y\subseteq[0,1]$ is closed
and $T_{a}$-invariant, then $\dim X\times Y=\dim X+\dim Y$ for all
sets $X$.
\begin{prop}
\label{prop:relating-dimY-and-dimFY}If $Y\subseteq[0,1]$ is closed
and $T_{b}$-invariant, then $\{b^{-t}\,:\,t\in F_{b,Y}(Y)\}\subseteq Y-Y\bmod1$
and $\dim F_{b,Y}(Y)\leq2\dim Y$.
\end{prop}

\begin{proof}
We can assume that $F_{b,Y}(Y)\neq\emptyset$ (i.e. $Y$ is infinite),
since otherwise the statement is trivial. Let $t\in F_{b,Y}(Y)$.
We can assume that $t\neq0$, since if it is then $b^{0}\in Y-Y\bmod1$
trivially. By definition, there exist points $y'_{n},y''_{n}\in Y$
such that $\{y'_{n}-y''_{n}\}_{b}\rightarrow t$. This means that
$\log_{b}(y'_{n}-y''_{n})\rightarrow t\bmod1$. Assuming without loss
of generality that $y'_{n}\geq y''_{n}$. there exist $k_{n}\in\mathbb{N}$
and $0\leq t_{n}<1$ such that $y'_{n}-y''_{n}=b^{-k_{n}-t_{n}}$
and $t_{n}\rightarrow t\bmod1$. Since $t\neq0$, this means convergence
is also in $\mathbb{R}$ once we identify $t_{n},t$ with elements
of $[0,1)$, which we do henceforth. Since $y'_{n}-y''_{n}<b^{-k_{n}}$,
we conclude that $T_{b}^{k_{n}}y''_{n}-T_{b}^{k_{n}}y'_{n}=b^{-t_{n}}\rightarrow b^{-t}$.
Passing to a subsequence we can assume that $T_{b}^{k_{n}}y'_{n}\rightarrow y'$
and $T_{b}^{k_{n}}y''_{n}\rightarrow y''$, hence $b^{-t}=y''-y'\in Y-Y$. 

For the second statement, the map $t\mapsto b^{-t}$ from $F_{b,Y}(Y)$
into $Y-Y$ is bi-Lipschitz to its image, so it preserves dimension.
In particular its image, which is a subset of $Y-Y$, has the same
dimension as $F_{b,Y}(Y)$. This gives $\dim F_{b,Y}(Y)\leq\dim(Y-Y)$.
Since $Y-Y$ is the image of $Y\times Y$ under the Lipschitz map
$(x,y)\mapsto x-y$, we have 
\[
\dim F_{b,Y}(Y)\leq\dim(Y\times Y)=2\dim Y\qedhere
\]
\end{proof}
A real function $f$ is to piecewise satisfy a property $P$ if there
is a discrete set $E$ such that for all $x\in\dom f\setminus E$,
the function satisfies $P$ in some neighborhood of $x$.
\begin{prop}
\label{prop:relating-dimFX-and-dimFfX}Let $X\subseteq\mathbb{R}$,
let $f\in C^{2}(\mathbb{R})$ be a diffeomorphism, and $Y=f(X)$.
Then $\dim F_{b,Y}(Y)\geq\dim F_{b,X}(X)-\bdim X$. The same holds
if $f$ is a piecewise $C^{2}$-diffeomorphism that is differentiable
on $X$.
\end{prop}

\begin{proof}
Let $g=f^{-1}:Y\rightarrow X$ and 
\[
Z=\bigcup_{y\in Y}\left(F_{b,Y}(y)\times\{\log_{b}|g'(y)|\}\right)\subseteq F_{b,Y}(Y)\times\mathbb{R}
\]
Writing $\pi(y,t)=y-t\bmod1$, we have $F_{b,X}(X)=\pi(Z)$, and since
$\pi$ is Lipschitz, 
\[
\dim F_{b,X}(X)\leq\dim Z
\]
On the other hand $Z\subseteq F_{b,Y}(Y)\times\log_{b}|g'(Y)|$, so
\begin{eqnarray*}
\dim Z & \leq & \dim F_{b,Y}(Y)+\bdim g'(Y)\\
 & \leq & \dim F_{b,Y}(Y)+\bdim Y\\
 & = & \dim F_{b,Y}(Y)+\bdim X
\end{eqnarray*}
where in the first line we used the fact that $\log$ preserves dimension,
the next inequality used the fact that $g'$ is $C^{1}$, hence piecewise
Lipschitz, so $\bdim g'(Y)\leq\bdim Y$, and the last inequality is
because $g$ is bi-Lipschitz, and hence $\dim Y=\dim g(X)=\dim X$.
Combining the inequalities gives the first claim. 

For the second statement, consider a finite partition of $\mathbb{R}$
into intervals $\{I_{i}\}_{i=1}^{N}$ such that $f|_{I_{i}}:I_{i}\rightarrow f(I_{i})$
is a $C^{2}$-diffeomorphism except possibly at the endpoints, which
then lie outside of $X$. Let $X_{i}=X\cap I_{i}$ and $Y_{i}=f(X_{i})$.
Since $X=\bigcup X_{i}$ and $Y=\bigcup Y_{i}$ and since there are
finitely many sets in these unions, we have $\bdim X=\max_{i}\bdim\dim X_{i}$
and $\bdim Y=\max_{i}\bdim Y_{i}$. By the first part, $\dim F_{b,Y_{i}}(Y_{i})\geq\dim F_{b,X_{i}}(X_{i})-\dim X_{i}$.
Since $X=\bigcup X_{i}$ (a disjoint union) and $Y=\bigcup Y_{i}$,
and both unions are finite, we have $F_{b,X}(X)=\bigcup F_{b,X_{i}}(X_{i})$
and $F_{b,Y}(Y)=\bigcup F_{b,Y_{i}}(Y_{i})$, and again the box dimension
of each set is given by the maximum of the dimensions of the sets
in the union. Thus, 
\begin{eqnarray*}
\dim F_{b,Y}(Y) & = & \max_{i}\dim F_{b,Y_{i}}(Y_{i})\\
 & \geq & \max_{i}(\dim F_{b,X_{i}}(X_{i})-\bdim X_{i})\\
 & \geq & \max_{i}\dim F_{b,X_{i}}(X_{i})-\bdim X\\
 & = & \dim F_{b,X}(X)-\bdim X\qedhere
\end{eqnarray*}
\end{proof}

\section{\label{sec:Proof-of-joint-invariance}Proof of Theorem \ref{thm:joint-invariant-sets}}

Let $X\subseteq[0,1]$ or $\mathbb{R}/\mathbb{Z}$ be closed, infinite,
$T_{a}$-invariant, and let $f$ be a $C^{2}$-diffeomorphism of $\mathbb{R}$
or $\mathbb{R}/\mathbb{Z}$ such that $Y=f(X)$ is a $T_{b}$-invariant
set for some $b\not\sim a$. 

Suppose first that $\dim X>0$. Then $h_{top}(X,T_{a})=\log a\cdot\dim X>0$.
By the variational principle we can find a $T_{a}$-invariant and
ergodic probability measure $\mu$ on $X$ with $h(\mu,T_{a})>0$.
Then by \cite[Theorem 1.10]{HochmanShmerkin2015-equidistribution-from-fractal-measures},
for $\mu$-a.e. $x$, the point $f(x)$ equidistributes under $T_{b}$
for Lebesgue measure, and in particular the $T_{b}$-orbit of $f(x)$
is dense. Since $f(X)\subseteq Y$ this means that $Y=[0,1]$. Since
$f$ is a diffeomorphism, $X$ must also be an interval, and the only
interval in $[0,1]$ invariant under $T_{a}$ is $[0,1]$.

It remains to show that $\dim X>0$. 
\begin{claim}
$F_{b,X}(X)=[0,1]$.
\end{claim}

\begin{proof}
$X$ is closed and infinite it contains an accumulation point $x_{0}\in X$,
whence 
\[
E=F_{b,X}(x_{0})\neq\emptyset
\]
Applying $T_{a}$ we have 
\[
F_{b,T_{a}^{n}X}(T_{a}^{n}x_{0})\supseteq E-n\log_{b}a\bmod1
\]
Since $X$ is $T_{a}$-invariant, 
\[
F_{b,X}(X)\supseteq\bigcup_{n\in\mathbb{N}}F_{b,X}(T_{a}^{n}x_{0})\supseteq E-\{n\log_{b}a\}_{n\in\mathbb{N}}\bmod1
\]
But $b\not\sim a$, so the set $\{n\log_{b}a\}_{n\in\mathbb{N}}$
is dense modulo $1$, hence also $E-\{n\log_{b}a\}_{n\in\mathbb{N}}$
and $F_{b,X}(X)$ are dense modulo 1. Since $F_{b,X}(X)$ is closed,
$F_{b,X}(X)=[0,1]$.
\end{proof}
To prove $\bdim X>0$, suppose by way of contradiction that $\bdim X=0$.
Now, $Y=f(X)$, so by Proposition \ref{prop:relating-dimFX-and-dimFfX},
\[
\dim F_{b,Y}(Y)\geq\dim(F_{b,X}X)-\bdim X=1-\dim X>0
\]
and so by Proposition \ref{prop:relating-dimY-and-dimFY}, $\dim Y\geq\frac{1}{2}\dim F_{b,Y}(Y)>0$.
But $f$ is bi-Lipschitz, so $\dim f(X)=\dim X=\dim Y>0$, as desired.

\section{\label{sec:Proof-of-partial-invariance}Proof of Theorem \ref{thm:partial-invariance-under-f}}

\textbf{Proof of part (\ref{transitive-case:minimal})}: Let $X\subseteq[0,1]$
be an infinite, minimal $T_{a}$-invariant set. Given $x,y\in X$
there is a sequence $n_{k}\rightarrow\infty$ such that we have $T_{a}^{n_{k}}x\rightarrow y$.
Since 
\[
F_{a,X}(x)\subseteq F_{a,T_{a}^{n_{k}}X}(T_{a}^{n_{k}}x)=F_{a,X}(T_{a}^{n_{k}}x)
\]
for all $k$, by semi-continuity of $F_{a,X}(\cdot)$ we conclude
that $F_{a,X}(x)\subseteq F_{a,X}(y)$. Since $x,y\in X$ were arbitrary,
$F_{a,X}(x)$ is independent of $x\in X$. We denote this set by $E$.
Since $X$ is infinite there is an accumulation point $x_{0}\in X$,
so $E=F_{a,X}(x_{0})\neq\emptyset$. 

Suppose that $I$ is a non-empty open interval and $f:I\rightarrow\mathbb{R}$
a $C^{1}$-embedding such that $f(X\cap I)\subseteq X$. Let $x\in X\cap I$,
we must show that $\alpha=\log_{a}|f'(x)|\in\mathbb{Q}$. Indeed suppose
$\alpha\notin\mathbb{Q}$. Then, writing $y=f(x)$, by the last paragraph
we have 
\begin{eqnarray*}
E & = & F_{a,X}(y)\\
 & \supseteq & F_{a,f(X)}(y)\\
 & = & F_{a,X}(x)-\alpha\bmod1\\
 & = & E+\alpha\bmod1
\end{eqnarray*}
Thus  $E$ is closed, non-empty and invariant under $t\mapsto t+\alpha\bmod1$.
Since $\alpha$ is irrational this implies $E=[0,1]$.

Now, for $x\in X$ we have $F_{a,X}(x)\subseteq\log_{a}(X-X)$ by
Proposition \ref{prop:relating-dimY-and-dimFY}. Thus $\log_{b}(X-X)$
has non-empty interior, so the same is true of $X-X$, and since $X-X\bmod1$
is $T_{a}$-invariant this implies that 
\[
X-X=[0,1]\bmod1
\]
But, since $X$ is minimal, this is impossible by \cite[Theorem III.1]{Furstenberg67}.

\textbf{Proof of part (\ref{transitive-case:restricted})}: We recall
that a set in a complete metric space is called residual it it contains
a dense $G_{\delta}$ subset. By Baire's theorem the intersection
of countably many residual sets is residual, and a residual set which
is also an $F_{\sigma}$ set contains a dense open set.

Suppose that $X$ is perfect and transitive, and self-restricted,
and that $f(X)=X$ for some local diffeomorphism $f\in C^{1}$. The
proof is very similar to the minimal case. Indeed, since $X$ is perfect,
$F_{a,X}(x)\neq\emptyset$ for all $x\in X$, in particular at transitive
points. By the same argument as above, $F_{a,X}(x)$ takes the same
value for all transitive points $x\in X$. Let $W\subseteq X$ denote
the set of transitive points, which is well-known to be is a residual
set in $X$. 

Let $I\subseteq[0,1]$ be a non-trivial interval such that, writing
$J=f(I)$, the restriction $f:X\cap I\rightarrow X\cap J$ is bijective.
Clearly $W\cap I$ is residual in $X\cap I$. At the same time, $f:X\cap I\rightarrow X\cap J$
is a homeomorphism, and $W\cap J$ is residual in $X\cap J$, so $f^{-1}(W\cap J)$
is residual in $X\cap I$. Consequently, $W\cap f^{-1}(W)\cap X\cap I$
is residual in $X\cap I$. Now, for any $x$ in this set, both $x$
and $f(x)$ are transitive. Arguing as in the minimal case we conclude
that if $\log_{a}f'(x)\notin\mathbb{Q}$ then $F_{a,X}(x)=[0,1]$,
hence $X-X=[0,1]\bmod1$. But this is impossible by restrictedness
of $X$.

By the last paragraph, the set of $x\in X$ such that $f'(x)\in\{a^{s}\,:\,s\in\mathbb{Q}\}$
is residual, i.e. contains a dense $G_{\delta}$. On the other hand
this set is just $(f')^{-1}(\{a^{s}\,:\,s\in\mathbb{Q}\})=\bigcup_{s\in\mathbb{Q}}(f')^{-1}(a^{s})$,
and by continuity of $f'$ this is an $F_{\sigma}$-set. By Baire's
theory applied in the compact metric space $X$, the set in question
must contain a dense and open (relative to $X$) set. 

\textbf{Real-analytic case of (\ref{transitive-case:minimal}) and
(\ref{transitive-case:restricted})}:\textbf{ }It remains to note
that if $f$ is real-analytic, then our conclusion shows that $f'(x)$
belongs to the countable set $\{a^{s}\,:\,s\in\mathbb{Q}\}$ for an
uncountable number of $x$, hence $f'$ takes on some rational power
$a^{s}$ of $a$ on a convergent sequence, and being itself real-analytic,
$f'\equiv a^{s}$. Thus $f$ is affine, and has the stated form.

\textbf{Proof of main statement of theorem}:\textbf{ }Let $X\subseteq[0,1]$
be a closed perfect, transitive and non-trivial $T_{a}$-invariant
set and $f\in C^{1}$ piecewise curved. If $\dim X=0$ then $X$ is
self-restricted, so $f(X)=X$ implies that $f'(x)\in\{a^{s}\,:\,s\in\mathbb{Q}\}$
for uncountably many $x\in X$; this is impossible because $f'$ is
piecewise strictly monotone, and so takes every value at most countably
many times. 

It remains to deal with the case that $\dim X>0$. By the variational
principle, there exists a $T_{a}$-ergodic probability measure $\mu$
on $X$ with $\dim\mu=\dim X$. If $f(X)=X$ then $f\mu$ is also
supported on $X$, and by Theorem \ref{thm:measures}, we obtain (by
averaging $T_{a}^{n}f\mu$ along some subsequence of times) a measure
on $X$ of dimension $>\dim X$, which is impossible.

\textbf{Proof of part} (\ref{transitive case:affine}): Let $f(x)=rx+t$
with $r\neq0$. If $\dim X=0$, let $x\in X$ with $F_{a,X}(x)\neq\emptyset$.
Then $F_{a,X}(f^{n}x)=F_{a,X}(x)-n\log_{a}|r|\bmod1$, and if $\log_{a}|r|\notin\mathbb{Q}$
this means that $F_{a,X}(X)$ contains a dense subset of $[0,1]$
and so is equal to $[0,1]$. By Proposition \ref{prop:relating-dimY-and-dimFY}
this is inconsistent with $X$ being self-restricted. Thus $\log_{a}|r|\in\mathbb{Q}$
as claimed.

In the case $\dim X>0$, we provide a similar proof in the next section
using measure theoretic tools.

\section{Proof of Theorem \ref{thm:measures} and Theorem \ref{thm:partial-invariance-under-f}
part (\ref{transitive case:affine})}

We provide a proof sketch, since a full proof would be lengthy, and
the results in \cite{Shmerkin2016,Wu2016} can be used to give alternative
proofs of the application to Theorems \ref{thm:joint-invariant-sets}
and \ref{thm:partial-invariance-under-f}. We rely heavily on the
scenery flow methods from \cite{HochmanShmerkin2012,HochmanShmerkin2015-equidistribution-from-fractal-measures}
and additive combinatorics methods from \cite{Hochman2014-self-similar-sets-with-overlaps}
and refer the reader to those papers for definitions and notation. 

Let $\mu$ be a $T_{a}$-ergodic measure $\mu$  of dimension $0<s<1$
and $f\in C^{1}$ piecewise curved. Let $\nu=f\mu$ and let $\nu'=\lim\frac{1}{N_{k}}\sum_{n=0}^{N_{k}-1}T_{a}^{n}(f\mu)$
for some sequence $N_{k}\rightarrow\infty$. We claim that there exists
$\varepsilon=\varepsilon(\mu)>0$ such that $\dim\nu'>\dim\mu+\varepsilon$.

It is well known (e.g. \cite{Hochman2010}) that $\mu$ generates
an ergodic fractal distribution (EFD; \cite[Definition 1.2]{Hochman2010b})
$P$ supported on measures of dimension $s$, which are supported
(up to a bounded scaling and translation) on $X$. See e.g. \cite[Section 2.2]{Hochman2010}.
Because $X$ is porous, $P$-a.e. measure is $(1-\varepsilon')$ -entropy
porous along any sequence $[n^{1+\tau}]$ of scales, in the sense
of \cite[Section 6.3]{Hochman2017-some-problems-on-the-boundary-of-fractal-geometry},
for some $\varepsilon'$ depending only on $\dim X$.

Since $f$ is a piecewise diffeomorphism, for $\mu$-a.e. $x$ the
measure $\nu=f\mu$ $\log a$-generates $S_{\log f'(x)}^{*}P$ at
$y=f(x)$ (see e.g. the proof of \cite[Lemma 4.16]{HochmanShmerkin2015-equidistribution-from-fractal-measures}).

Let $\pi$ denote the (partially defined) operation of restricting
a measure to $[0,1]$ and normalizing it to a probability measure.
It now follows, as in \cite[Theorem 5.1]{HochmanShmerkin2015-equidistribution-from-fractal-measures},
that if $\frac{1}{N_{k}}\sum_{n=0}^{N_{k}-1}T_{a}^{n}\nu\rightarrow\nu'$
then there is an auxiliary probability space $(\Omega,\mathcal{F},Q)$
and functions $y,t:\Omega\rightarrow\mathbb{R}$ and $\eta:\Omega\rightarrow\mathcal{P}([0,1])$,
such that $t_{\omega}$, $\omega\sim Q$ is distributed like $\log f'(x)$,
$x\sim\mu$, and $\eta_{\omega}$, $\omega\sim Q$ is distributed
according to $P$, and such that
\[
\nu'=\int\pi(S_{t_{\omega}}(\eta_{\omega}*\delta_{y_{\omega}}))dQ(\omega)
\]
(the formula above differs from \cite[Theorem 5.1]{HochmanShmerkin2015-equidistribution-from-fractal-measures}
in the scaling $S_{t_{\omega}}$ of the integrand; the scaling comes
from the fact that at $\nu$-a.e. point $y=f(x)$, the measure $\nu$
generates $S_{t}^{*}P$, where $t=\log f'(x)$. Also, the theorem
in \cite{HochmanShmerkin2015-equidistribution-from-fractal-measures}
refers to $\nu'$ arising from the orbit of a single $\nu$-typical
point, not as above, but in fact the averaged version above follows
from the pointwise one).

Re-interpreting the last equation and the properties of $t,\eta$
stated before it, we find that $\nu'$ can be generated in the following
way: choose $x$ according to $\mu_{\omega}$, independently choose
a $P$-typical measure $\eta$, scale $\eta$ by $f'(x)$, and translate
by a random amount $y$ (whose distribution depends on $x,\eta$).
Changing the order with which we choose $x$ and $\eta$, we find
that for $P$-typical $\eta$ there is a probability measure $\theta_{\eta}$
on the group of affine maps of $\mathbb{R}$ such that we can represent
$\nu'$ as 
\begin{eqnarray*}
\nu' & = & \int\int\pi(T\eta)\,d\theta_{\eta}(T)\,dP(\eta)\\
 & = & \int\pi(\theta_{\eta}*\eta)\,dP(\eta)
\end{eqnarray*}
Furthermore, choosing $T\sim\theta$, the distribution of the contraction
ratio of $T$ is the same as the distribution of $f'(x)$ for $x\sim\mu$.
Since $f$ was assumed piecewise curved, $f'$ is locally bi-Lipschitz,
so the image of $\mu$ under $x\mapsto f'(x)$ has the same dimension
as $\mu$, so by the above, $\dim\theta_{\eta}\geq\dim\nu_{\omega}=s$.
Also, recall that $\eta$ has uniform entropy dimension $s$ (in the
sense of \cite[Definition  5.1]{Hochman2014-self-similar-sets-with-overlaps};
this is an immediate consequence of the definition of the measures
$\eta$, of the definition of Kolmogorov-Sinai entropy, and of the
ergodic theorem). It follows from \cite[Theorem 9]{Hochman2017-some-problems-on-the-boundary-of-fractal-geometry}
that $\hdim\theta_{\eta}*\eta\geq\dim\eta+\varepsilon=s+\varepsilon$
for some $\varepsilon=\varepsilon(s)$, hence
\[
\hdim\nu'\geq\int\hdim\theta_{\eta}*\eta\,dP(\eta)\geq s+\varepsilon
\]
which is what we wanted to show.

We now turn to the case that $f(x)=rx+t$ is affine, $f(X)\subseteq X$
and $\dim X>0$. Fix a dimension-maximizing $T_{a}$-invariant and
ergodic measure $\mu$ on $X$. Choose a typical point $x$ and consider
the EFD $P_{n}$ $\log a$-generated at $f^{n}x$. Evidently this
is $P_{n}=S_{n\log_{a}r}P_{0}$. Assuming $\log_{a}r\notin\mathbb{Q}$,
we can average and pass to a weak{*} limit, and find that $X$ supports
a measure of the form
\[
\nu=\int\int_{0}^{\log a}\pi(S_{t}(\eta*\delta_{y_{\eta,t}}))dtdP(\eta)
\]
We now again apply \cite[Theorem 9]{Hochman2017-some-problems-on-the-boundary-of-fractal-geometry}
to conclude that $\dim\nu>\dim\mu=\dim X$, a contradiction, which
shows that $\log_{a}r\in\mathbb{Q}$.

\bibliographystyle{plain}
\bibliography{bib}

\lyxaddress{Email: mhochman@math.huji.ac.il \\
Address: Einstein Institute of Mathematics, Edmond J. Safra campus,
Jerusalem 91904, Israel}
\end{document}